\documentclass[a4paper,12pt]{article}
\usepackage[text={6.5in,9.5in},centering]{geometry}
\usepackage{graphicx,url,subfig}
\RequirePackage[OT1]{fontenc}
\RequirePackage{amsthm,amsmath,amssymb,amscd}
\RequirePackage[numbers]{natbib}
\RequirePackage[colorlinks,citecolor=blue,urlcolor=blue]{hyperref}
\usepackage{enumerate}
\usepackage{mathrsfs}
\usepackage{datetime}
\usepackage{etex}
\usepackage[normalem]{ulem} % either use this (simple) or
\usepackage{soul} % use this (many fancier options)
\usepackage{pgfplots}
\makeatother
\numberwithin{equation}{section}
\allowdisplaybreaks

\theoremstyle{plain}
\newtheorem{theorem}{Theorem}[section]

\newtheorem{lemma}[theorem]{Lemma}
\newtheorem{proposition}[theorem]{Proposition}

\theoremstyle{definition}
\newtheorem{definition}[theorem]{Definition}

\newtheorem{remark}[theorem]{Remark}

\newcommand{\E}{\mathbb{E}}

\newcommand{\R}{\mathbb{R}}

\newcommand{\RR}{\mathbb{R}}

\newcommand{\ot}{[0,t]}
\newcommand{\ott}{[0,T]}

\newcommand{\1}{{\bf 1}}

\newcommand{\be}{\mathbf{E}}

\newcommand{\bp}{\mathbf{P}}

\newcommand{\cf}{\mathcal F}

\newcommand{\cv}{\mathcal V}

\newcommand{\HH}{\mathfrak H}
\newcommand{\al}{\alpha}

\newcommand{\ep}{\varepsilon}

\newcommand{\gga}{\Gamma}
\newcommand{\ka}{\kappa}
\newcommand{\la}{\lambda}
\newcommand{\laa}{\Lambda}
\newcommand{\om}{\omega}
\newcommand{\oom}{\Omega}

\newcommand{\si}{\sigma}

\newcommand{\vp}{\varphi}

\newcommand{\lp}{\left(}
\newcommand{\rp}{\right)}
\newcommand{\lc}{\left[}
\newcommand{\rc}{\right]}
\newcommand{\lcl}{\left\{}
\newcommand{\rcl}{\right\}}

\allowdisplaybreaks[4]

\newtheorem{problem}[theorem]{Problem}
\theoremstyle{remark}

\let\Section=\section
\def\section{\setcounter{equation}{0}\Section}

\def\RR{\mathbb{R}}

\def\la{{\lambda}}
\def\si{{\sigma}}

\def \eref#1{\hbox{(\ref{#1})}}

\title{Parabolic Anderson model with rough dependence in space
 }
\author{Yaozhong Hu\thanks{Y. Hu is partially supported by a grant from the Simons Foundation $\#$209206.}, Jingyu Huang, Khoa L\^e, David Nualart\thanks{ D. Nualart is supported by the NSF grant DMS1208625 and the ARO grant FED0070445.}, Samy Tindel\thanks{S. Tindel is supported by the NSF grant  DMS1613163}
\date{\vspace{0em}\small %\today
  }
}

\begin{document}
\maketitle
\begin{abstract}
This paper studies the one-dimensional parabolic Anderson model driven by a Gaussian noise which is white in time and has the covariance of a fractional Brownian motion with Hurst parameter $H \in (\frac{1}{4}, \frac{1}{2})$ in the space variable. We derive the Wiener chaos expansion of the solution and a Feynman-Kac formula for the moments of the solution. These results allow us to establish sharp lower and upper asymptotic bounds for the $n$th moment of the solution. 


\end{abstract}%\newpage

% \tableofcontents
\setlength{\parindent}{1.5em}

%%%%%%%%%%%%%%%%%%%%%%%%%%%%%%%%%%%%%%%%%%%%%%
%%%%% MAIN: The chapters of the thesis
%%%%%%%%%%%%%%%%%%%%%%%%%%%%%%%%%%%%%%%%%%%%%%

% \mainmatter
% \graphicspath{{../figs/}}

\section{Introduction}
 A recent paper \cite{HHLNT} studies the stochastic heat equation for $(t,x) \in (0,\infty)\times \RR$
\begin{equation}\label{eq:SHE sigma}
\frac{\partial u}{\partial t}=\frac{\kappa}{2}\frac{\partial ^2 u}{\partial x
^2}+\sigma(u)\,\dot W\,,
\end{equation}
where $\dot{W}$ is a centered Gaussian noise which is white in time and behaves as fractional Brownian motion with Hurst parameter $1/4 < H < 1/2$ in space, and $\sigma$ may be a nonlinear function with some smoothness.

However, the specific case $\sigma(u)=u$, i.e.
\begin{equation}\label{spde}
\frac{\partial u}{\partial t}=\frac{\kappa}{2}\frac{\partial ^2 u}{\partial x
^2}+u\,\dot W
\end{equation}
deserves some specific treatment due to its simplicity. Indeed, this linear equation turns out to be a continuous version of the parabolic Anderson model, and is related to challenging systems in random environment like KPZ equation \cite{Ha,BeC} or polymers~\cite{AKQ,BTV}. The localization and intermittency properties \eqref{spde} have thus been thoroughly studied for equations driven by a space-time white noise (see \cite{Kh} for a nice survey), while a recent trend consists in extending this kind of result to equations driven by very general Gaussian noises \cite{Ch14,HHNT,HN,HNS}, but the rough noise $\dot W$ presented here is not covered by the aforementioned references. 

To fill this gap, we first go to the existence and uniqueness problem. Although the existence and uniqueness of the solution in the general
nonlinear  case \eqref{eq:SHE sigma} has been established in \cite{HHLNT}, in this linear case \eqref{spde}, one can implement a rather simple procedure involving Fourier transforms.  %as well as a chaos expansion technique, 
%in order to achieve existence and uniqueness of the solution to~\eqref{spde}. 
Since this point of view is interesting in its own right and is short enough, we develop it in Subsection \ref{sec:picard}. %Moreover in this case we can consider different initial conditions. 
In Subsection \ref{subsec: chaos}, we study the random field solution using chaos expansion. Following the approach introduced in \cite{HN,HHNT},  we obtain an explicit formula for the kernels of the Wiener chaos expansion and we show its convergence, and thus obtain the existence and uniqueness of the solution. It worths noting these methods treat different classes of initial data which are more general than in \cite{HHLNT} and different from \cite{BJQ}.

We then move to a Feynman-Kac type representation for the moments of the solution.  In fact, we  cannot  expect  a Feynman-Kac formula for the solution, because the  covariance is rougher than the space-time white noise case, and this type of formula requires smoother covariance structures  (see, for instance,  \cite{HNS}). However,  by means of Fourier analysis techniques as in \cite{HN,HHNT}, we are able to obtain a Feynman-Kac formula for the moments that involves a fractional derivative of the Brownian local time.

Finally, the previous considerations allow us to handle, in the last section of the paper,  the intermittency properties of the solution.   More precisely,  we show  sharp lower bounds for the moments of the solution of the form $\be [u(t,x)^n]\ge\exp(C n^{1+\frac 1H} t)$, for all $t\ge 0$, $x\in \R $ and $n\ge 2$,  where $C$ is independent of 
$t\ge 0$, $x\in \R $ and $n$.   These bounds entail the intermittency phenomenon and match the corresponding estimates for the case $H>\frac 12$ obtained in \cite{HHNT}. 
After the completion of this work, three of the authors have studied the parabolic Anderson model in more details in \cite{HLN}. In particular, existence and uniqueness results  are extended for a wider class of initial data, exact long term asymptotics for the moments of the solution are obtained. 

\section{Preliminaries}
%Our noise $W$ can be seen as a Brownian motion with values in an infinite dimensional Hilbert space. 

Let us start by introducing our basic notation on Fourier transforms
of functions. The space of   Schwartz functions is
denoted by $\mathcal{S}$. Its dual, the space of tempered distributions, is $\mathcal{S}'$.  The Fourier
transform of a function $u \in \mathcal{S}$ is defined with the normalization
\[ \mathcal{F}u ( \xi)  = \int_{\mathbb{R}} e^{- i
   \xi  x } u ( x) d x, \]
so that the inverse Fourier transform is given by $\mathcal{F}^{- 1} u ( \xi)
= ( 2 \pi)^{- 1} \mathcal{F}u ( - \xi)$.

 Let $ \mathcal{D}((0,\infty)\times \R)$ denote the space  of real-valued infinitely differentiable functions with compact support on $(0, \infty) \times \R$.
Taking into account the spectral representation of the covariance function of the fractional Brownian motion in the case $H<\frac 12$
proved in \cite[Theorem 3.1]{PT}, we represent  our noise $W$   by a zero-mean Gaussian family $\{W(\vp) ,\, \vp\in
\mathcal{D}((0,\infty)\times \R)\}$ defined on a complete probability space
$(\Omega,\cf,\bp)$, whose covariance structure
is given by
\begin{equation}\label{eq:cov1}
\be\lc W(\vp) \, W(\psi) \rc
=  c_{1,H}\int_{\R_{+}\times\R}
\cf\varphi(s,\xi) \, \overline{\cf\psi(s,\xi)} \, |\xi|^{1-2H} \, ds  d\xi,
\end{equation}
where the Fourier transforms $\cf\varphi,\cf\psi$ are understood as Fourier transforms in space only and
\begin{equation}\label{eq:expr-c1H}
c_{1,H}= \frac 1 {2\pi} \Gamma(2H+1)\sin(\pi H)  \,.
\end{equation}

We 
%define the class of functions 
denote by $\HH $ the  Hilbert space obtained by  completion of $ \mathcal{D}((0,\infty)\times \R)$ with respect to the inner product
\begin{equation}\label{eq: H_0 element H prod}
  \langle\varphi, \psi \rangle_{ \HH}=c_{1, H}\int_{\RR_+\times \RR}\mathcal{F}\varphi(s,\xi)\overline{\mathcal{F}\psi(s,\xi)}|\xi|^{1-2H }d\xi ds\,.
  \end{equation}     
%\begin{equation}\label{eq: H inner prod}
%\langle \varphi,\psi\rangle_{ \HH}:=\ c_{2 ,  H }\int_{\RR_+\times \RR}D_-^{\frac 12-H }\varphi(s,x)D_-^{\frac 12-H}\psi(s,x)dsdx\,,
%\end{equation}
  %and  $ \mathcal{D}((0,\infty)\times \R)$ is dense in  $\HH$.
  The next proposition is from Theorem 3.1 and Proposition 3.4 in \cite{PT}. 
  \begin{proposition} \label{prop: H} 
 If $\HH_0$ denotes the  class of functions   $\varphi  \in L^2( \RR_+\times \RR)$ such that 
 $$ \int_{\RR_+\times \RR} |\mathcal{F}\varphi(s,\xi)|^2|\xi|^{1-2H}d\xi ds < \infty\,, $$ then $ \HH_0$ is not complete  and the inclusion
  $\HH_0 \subset \HH$ is strict. %Also for any $\varphi,\psi \in \HH_0$, 
\end{proposition}

We recall   that the Gaussian family $W$ can be extended to $\HH$ and this produces an isonormal Gaussian process, {%\color {red} 
for which Malliavin calculus can be applied. }%where $\HH$ is the Hilbert space  introduced in Proposition \ref{prop: H}. 
We refer to~\cite{Nua}   and  \cite{hubook}  
for a detailed account of the Malliavin calculus with respect to a
Gaussian process. On our Gaussian space, the  smooth and cylindrical
random variables $F$ are of the form
\begin{equation*}
F=f(W(\phi_1),\dots,W(\phi_n))\,,
\end{equation*}
with $\phi_i \in \HH$, $f \in C^{\infty}_p (\R^n)$ (namely $f$ and all
its partial derivatives have polynomial growth). For this kind of random variable, the derivative operator $D$ in the sense of Malliavin calculus is the
$\HH$-valued random variable defined by
\begin{equation*}
DF=\sum_{j=1}^n\frac{\partial f}{\partial
x_j}(W(\phi_1),\dots,W(\phi_n))\phi_j\,.
\end{equation*}
The operator $D$ is closable from $L^2(\Omega)$ into $L^2(\Omega;
\HH)$  and we define the Sobolev space $\mathbb{D}^{1,2}$ as
the closure of the space of smooth and cylindrical random variables
under the norm
\[
\|DF\|_{1,2}=\sqrt{\be [F^2]+\be [\|DF\|^2_{\HH}  ]}\,.
\]
We denote by $\delta$ the adjoint of the derivative operator (called divergence operator) given
by the duality formula
\begin{equation}\label{dual}
\be  \lc \delta (u)F \rc =\be  \lc \langle DF,u
\rangle_{\HH}\rc ,
\end{equation}
for any $F \in \mathbb{D}^{1,2}$ and any element $u \in L^2(\Omega;
\HH)$ in the domain of $\delta$.

For any integer $n\ge 0$ we denote by $\mathbf{H}_n$ the $n$th Wiener chaos of $W$. We recall that $\mathbf{H}_0$ is simply  $\R$ and for $n\ge 1$, $\mathbf {H}_n$ is the closed linear subspace of $L^2(\Omega)$ generated by the random variables $\{ H_n(W(\phi)),\phi \in \HH, \|\phi\|_{\HH}=1 \}$, where $H_n$ is the $n$th Hermite polynomial.
For any $n\ge 1$, we denote by $\HH^{\otimes n}$ (resp. $\HH^{\odot n}$) the $n$th tensor product (resp. the $n$th  symmetric tensor product) of $\HH$. Then, the mapping $I_n(\phi^{\otimes n})= H_n(W(\phi))$ can be extended to a linear isometry between    $\HH^{\odot n}$ (equipped with the modified norm $\sqrt{n!}\| \cdot\|_{\HH^{\otimes n}}$) and $\mathbf{H}_n$.

Consider now a random variable $F\in L^2(\Omega)$ which is measurable with respect to the $\sigma$-field  $\cf$ generated by $W$. This random variable can be expressed as
\begin{equation}\label{eq:chaos-dcp}
F= \be \lc F\rc + \sum_{n=1} ^\infty I_n(f_n),
\end{equation}
where the series converges in $L^2(\Omega)$, and the elements $f_n \in \HH ^{\odot n}$, $n\ge 1$, are determined by $F$.  This identity is called the Wiener chaos expansion of $F$.

The Skorohod integral (or divergence) of a random field $u$ can be
computed by  using the Wiener chaos expansion. More precisely,
suppose that $u=\{u(t,x) , (t,x) \in \R_+ \times\R\}$ is a random
field such that for each $(t,x)$, $u(t,x)$ is an
$\cf_t$-measurable and square-integrable random  variable, here $\mathcal{F}_t$ is the $\sigma$ algebra generated by $W$ up to time $t$.
Then, for each $(t,x)$ we have a Wiener chaos expansion of the form
\begin{equation}  \label{exp1}
u(t,x) = \be \lc u(t,x) \rc + \sum_{n=1}^\infty I_n (f_n(\cdot,t,x)).
\end{equation}
Suppose   that $\be [\|u\|_{ \HH}^{2}]$ is finite.
Then, we can interpret $u$ as a square-integrable
random function with values in $\HH$ and the kernels $f_n$
in the expansion (\ref{exp1}) are functions in $\HH
^{\otimes (n+1)}$ which are symmetric in the first $n$ variables. In
this situation, $u$ belongs to the domain of the divergence operator (that
is, $u$ is Skorohod integrable with respect to $W$) if and only if
the following series converges in $L^2(\Omega)$
\begin{equation}\label{eq:delta-u-chaos}
\delta(u)= \int_0 ^\infty \int_{\R^d}  u(t,x) \, \delta W(t,x)
= W(\be [u]) + \sum_{n=1}^\infty I_{n+1} (\widetilde{f}_n),
\end{equation}
where $\widetilde{f}_n$ denotes the symmetrization of $f_n$ in all its $n+1$ variables.
We note here that if $\Lambda_{H}$ denotes the space of predictable processes $g$ defined on $\RR_+\times \RR$ such that almost surely $g\in \HH$ and $\be [\|g\|^2_{\HH}] < \infty$, the Skorohod integral of $g$ with respect to $W$ coincides with the It\^o integral defined in \cite{HHLNT}, also, we have the isometry 
\begin{equation}\label{int isometry}
\be \left [\left( \int_{\RR_+} \int_{\RR} g(s,x) W(ds,dx) \right)^2 \right] = \E \|g\|^2_{\HH}\,.
\end{equation}
Now we are ready to state the definition of the solution to equation \eqref{spde}. 

\begin{definition}
Let $u=\{u(t,x), 0 \leq t \leq T, x \in \mathbb{R}\}$ be a real-valued predictable stochastic process  such that for all $t\in[0,T]$ and $x\in\R$ the process $\{p_{t-s}(x-y)u(s,y) \1_{[0,t]}(s), 0 \leq s \leq t, y \in \mathbb{R}\}$ is Skorohod integrable, where $p_t(x)$ is the heat kernel on the real line related to $\frac{\ka}{2}\Delta$. We say that $u$ is a mild solution of \eqref{spde} if for all $t \in [0,T]$ and $x\in \mathbb{R}$ we have
\begin{equation}\label{eq:mild-formulation sigma}
u(t,x)= p_t*u_0(x) + \int_0^t \int_{\mathbb{R}}p_{t-s}(x-y)u(s,y) W(ds,dy) \quad a.s.,
\end{equation}
where the stochastic integral is understood  in the sense of Skorohod or It\^o.
\end{definition}

%We note that whenever $u\in \laa_{H}$ the integral $\delta(u)$ coincides with the It\^o integral defined in Proposition \ref{prop:intg-wrt-W}.

%\section{The Anderson model}\label{sec:Anderson}
%In this section we will study the special case of equation \eref{spde with sigma} when the function $\sigma$ is the identity function. This is  a continuous version of the so-called  parabolic Anderson model. In this case equation \eref{spde with sigma} is reduced to with deterministic initial condition $u(0,x)=u_0(x)$. This reduced form allows for some simplified versions of the existence-uniqueness theorems, and also \hb{for a} Feynman-Kac representation which is useful for intermittency estimates.

\section{Existence and uniqueness}\label{sec:anderson-exist-uniq}
% With some restrictions on the initial condition $u_0(x)$, the existence and uniqueness of the solution to this linear equation stems directly from Theorems~\ref{thm:uniqueness} and~\ref{thm:exist with sigma} in Section \ref{sec:existence}. 

In this section we prove the existence and uniqueness result for the solution to equation \eqref{spde} by means of two different methods: %which do not need the several function spaces introduced in Section \ref{sec:existence}: 
one is via Fourier transform and the other is via chaos expansion. %We include these methods here for two reasons: first,  they lead to proofs which are shorter and more elegant than in the case of a general coefficient $\si$;  secondly, the  assumptions on the initial condition  are different. 

\subsection{Existence and uniqueness via Fourier transform}\label{sec:picard}

In this subsection we discuss the existence and uniqueness of equation (\ref{spde}) using techniques of Fourier analysis.

%We recall that $\dot{H}^{\frac 12-H}$ is the class of functions  $f:    \RR \rightarrow \RR $ such that  there exists $  g\in L^2(   \RR)  $ such that $ f =I_-^{1/2-H}g$. 

Let    $\dot{H}^{\frac 12-H}_0$  be the set of functions $f\in L^2(\RR)$ such that $\int_\RR | \cf f(\xi)| ^2 |\xi|^{1-2H} d\xi <\infty$.
   This spaces is the time independent analogue to the space $\HH_0$ introduced in Proposition
\ref{prop: H}.  We know that %the inclusion  $\dot{H}^{\frac 12-H}_0 \subset  \dot{H}^{\frac 12-H}$ is strict and 
$\dot{H}^{\frac 12-H}_0$ is not complete with the seminorm  $ \left[ \int_\RR | \cf f(\xi)| ^2 |\xi|^{1-2H} d\xi \right] ^\frac 12$ (see \cite{PT}). However,  it is not difficult to check that the space   $\dot{H}^{\frac 12-H}_0$ is complete for the seminorm $\|f\|_{\cv(H)} ^2:=\int_\RR | \cf f(\xi)| ^2  (1+|\xi|^{1-2H} )d\xi$.
% \[
% \|f\|_{\cv(H)} ^2:=\int_\RR | \cf f(\xi)| ^2  (1+|\xi|^{1-2H} )d\xi.
% \]

 In the  next theorem  we show the existence and uniqueness result  assuming that
 the initial condition belongs to $\dot{H}^{\frac 12-H}_0$ and
 using estimates based on the Fourier transform in the space variable. To this purpose, we   introduce
 the space $\cv_T(H)$  as the completion of the set  of elementary   $\dot{H}^{\frac 12-H}_0$-valued stochastic processes $\{u(t,\cdot), t\in [0,T]\}$ with respect to the seminorm
\begin{equation}  \label{nuH}
\|u\|_{\cv_{T}(H)}^{2}:=\sup_{t\in [0,T]}   \be \| u(t,\cdot)\|_{\cv(H)}^{2}.
\end{equation}

We now state a convolution lemma.

\begin{proposition}\label{prop:convolution-fourier}
Consider a function $u_{0}\in \dot{H}^{\frac 12-H}_0$ and
$\frac{1}{4}<H<\frac{1}{2}$. For any  {$v\in\cv_{T}(H)$} we set
$\gga(v)=V$ in the following way:
\begin{equation*}
\gga(v):=V(t,x)=p_t*u_0(x) + \int_0^t \int_{\mathbb{R}}p_{t-s}(x-y) v(s,y) W(ds,dy),
\quad t\in[0,T], \, x\in\R.
\end{equation*}
Then $\gga$ is well-defined as a map from $\cv_{T}(H)$ to $\cv_{T}(H)$. Furthermore, there exist two positive constants $c_{1},c_{2}$ such that the following estimate holds true on $[0,T]$:
\begin{equation}\label{eq:intg-bnd-V-lin}
{ \|V(t,\cdot)\|_{\cv(H)}^{2} \le c_{1} \, \|u_0\|_{\cv(H)}^{2}
+c_{2}\int_0^t  (t-s)^{2H-3/2}
 \|v(s,\cdot)\|_{\cv(H)}^{2} \,
ds\,.}
\end{equation}
\end{proposition}

\begin{proof}
Let $v$ be a process in $\cv_{T}(H)$ and set $V=\gga(v)$.  We
focus on the bound \eqref{eq:intg-bnd-V-lin} for $V$. 

Notice that the Fourier transform of $V$ can be computed easily.
 {Indeed, setting $v_0(t,x)=p_t*u_0(x)$ and }invoking a stochastic
version of Fubini's theorem, we get
\begin{equation*}
\mathcal{F}V(t,\xi)=\mathcal{F}v_0(t,\xi)
+\int_0^t\int_{\mathbb{R}} \lp \int_{\R} e^{i x \xi} \, p_{t-s}(x-y) \, dx \rp
v(s,y)W(ds,dy)\,.
\end{equation*}
According to the expression of $\cf p_{t}$, we obtain
\begin{eqnarray*}
\mathcal{F}V(t,\xi)=\mathcal{F}v_0(t,\xi)+\int_0^t\int_{\mathbb{R}}e^{-i\xi
y} e^{-\frac{\kappa}{2}(t-s)\xi^2}v(s,y)W(ds,dy)\,.
\end{eqnarray*}
We now evaluate the quantity
$\be[\int_{\mathbb{R}}|\mathcal{F}V(t,\xi)|^2|\xi|^{1-2H}d\xi ]$ in
the definition of  $\|V\|_{\cv_{T}(H)}$  given by \eqref{nuH}.  We
thus write
\begin{multline*}
\be\lc \int_{\mathbb{R}}|\mathcal{F}V(t,\xi)|^2|\xi|^{1-2H}d\xi \rc
\leq 2 \, \int_{\mathbb{R}}|\mathcal{F}v_0(t,\xi)|^2|\xi|^{1-2H}d\xi \\
+2 \,  \int_{\mathbb{R}}\be\lc\Big|\int_0^t\int_{\mathbb{R}}e^{-i\xi
y}e^{-\frac{\kappa}{2}(t-s)\xi^2}v(s,y)W(ds,dy)\Big|^2 \rc |\xi|^{1-2H}d\xi
:= 2\lp I_{1} + I_{2} \rp \, ,
\end{multline*}
and we handle the terms $I_{1}$ and $I_{2}$ separately.

The term $I_1$ can be easily bounded by using that $u_0 \in\dot{H}^{\frac 12-H}_0$ and recalling $v_{0}=p_{t}*u_{0}$. That is,
\[
I_1 = \int_\RR| \mathcal{F}u_0(\xi) | ^2e^{-\kappa t|\xi|^2} |\xi|^{1-2H}
d\xi \le C \, \|u_{0}\|_{\cv(H)}^{2}.
\]
 We thus focus on the estimation of $I_{2}$, and we set $f_{\xi}(s,\eta)=e^{-i\xi
\eta}e^{-\frac{\kappa}{2}(t-s)\xi^2}v(s,\eta)$. Applying the isometry
property \eqref{int isometry} %together with the Fourier transform expression for $\|h\|_{\dot H^{\frac 12-H}}$ in~\eref{eq: H_0 element H prod}, 
we have:
\begin{equation*}
\be\lc\Big|\int_0^t\int_{\mathbb{R}}
e^{-i\xi y}e^{-\frac{\kappa}{2}(t-s)\xi^2}v(s,y)W(ds,dy)\Big|^2 \rc
=c_{1,H} \int_0^t \int_{\mathbb{R}}
\be\lc |\cf_{\eta}f_{\xi}(s,\eta) |^{2}\rc
|\eta|^{1-2H} \, ds d\eta,
\end{equation*}
where $\cf_{\eta}$ is  the Fourier transform with respect  to
$\eta$.  It is obvious %from the definition of Fourier transform
 that
the Fourier transform of $e^{-i\xi y} V(y)$ is $\mathcal{F}
V(\eta+\xi)$. Thus we have
\begin{align*}
I_{2}&= C\int_0^t\int_{\mathbb{R}}\int_{\mathbb{R}}e^{-\kappa(t-s)\xi^2}
\, \be\lc|\mathcal{F}v(s,\eta+\xi)|^2 \rc |\eta|^{1-2H}|\xi|^{1-2H}
\, d\eta d\xi ds\\
&= C\int_0^t\int_{\mathbb{R}}\int_{\mathbb{R}}e^{-\kappa(t-s)\xi^2} \,
\be\lc|\mathcal{F}v(s,\eta )|^2 \rc |\eta-\xi|^{1-2H}|\xi|^{1-2H} \,
d\eta d\xi ds\, .
\end{align*}
We now bound $|\eta-\xi |^{1-2H}$ by $|\eta|^{1-2H}+|\xi|^{1-2H}$,
which yields $I_{2}\le I_{21}+I_{22}$ with:
\begin{align*}
I_{21}&=C
\int_0^t
\int_{\mathbb{R}}\int_{\mathbb{R}} e^{-\kappa(t-s)\xi^2} \,
\be\lc|\mathcal{F}v(s,\eta)|^2 \rc
|\eta|^{1-2H}|\xi|^{1-2H} \, d\eta d\xi ds \\
I_{22}&=C\int_0^t\int_{\mathbb{R}}\int_{\mathbb{R}}e^{-\kappa(t-s)\xi^2}
\, \be\lc|\mathcal{F}v(s,\eta)|^2 \rc |\xi|^{2-4H} \, d\eta d\xi
ds\,.
\end{align*}
Performing the change of variable  {$\xi \rightarrow
(t-s)^{1/2}\xi$} and then trivially bounding the integrals of the
form $\int_{\R}|\xi|^{\beta} e^{-\kappa\xi^{2}} d\xi$ by constants,  we
end up with
\begin{align*}
I_{21}&\leq C
\int_0^t  (t-s)^{H-1}
\int_{\mathbb{R}}
\be\lc|\mathcal{F}v(s,\eta)|^2 \rc
|\eta|^{1-2H} \, d\eta  \, ds \\
I_{22}&\leq C
\int_0^t  (t-s)^{2H-3/2}
\int_{\mathbb{R}}
\be\lc|\mathcal{F}v(s,\eta)|^2 \rc
 \, d\eta  \, ds .
\end{align*}
Observe that for $H\in(\frac 14, \frac 12)$ the term $(t-s)^{2H-3/2}$ is more
singular than  $(t-s)^{H-1}$, but we still have $2H-\frac 32>-1$
(this is where  we need to  impose $H>1/4$).
Summarizing our consideration  up to now, we have thus obtained
\begin{multline}\label{eq:bnd-picard-1}
\int_{\mathbb{R}}\be\lc |\mathcal{F}V(t,\xi)|^2 \rc |\xi|^{1-2H}d\xi \\
\le C _{1,T} \, { \|u_{0}\|_{\cv(H)}^{2}} + C_{2,T} \int_{0}^{t}
(t-s)^{2H-3/2} \int_{\mathbb{R}} \be\lc|\mathcal{F}v(s,\xi)|^2 \rc
\, (1+ |\xi|^{1-2H})
 \, d\xi  \, ds ,
\end{multline}
for two strictly positive constants $C_{1,T},C_{2,T}$.

The term $\be[\int_{\mathbb{R}}|\mathcal{F}V(t,\xi)|^2 d\xi ]$ in
the definition of $\|V\|_{\cv_{T}(H)}$ can be bounded with the same computations as above, and we find
\begin{multline}\label{eq:bnd-picard-2}
\int_{\mathbb{R}}\be\lc |\mathcal{F}V(t,\xi)|^2 \rc \, d\xi \\
\le C_{1,T} \, {  \|u_{0}\|_{\cv(H)}^{2}} + C_{2,T}  \int_{0}^{t}
(t-s)^{H-1} \int_{\mathbb{R}} \be\lc|\mathcal{F}v(s,\xi)|^2 \rc \,
(1+ |\xi|^{1-2H})
 \, d\eta  \, ds ,
\end{multline}
Hence, gathering our estimates \eqref{eq:bnd-picard-1} and \eqref{eq:bnd-picard-2}, our bound \eqref{eq:intg-bnd-V-lin} is easily obtained, which finishes the proof.
\end{proof}

As in the forthcoming general case, Proposition \ref{prop:convolution-fourier} is the key to the existence and uniqueness result for equation \eqref{spde}.

\begin{theorem}\label{thm:exist-uniq-picard}
Suppose that $u_{0}$ is an element of $\dot{H}^{\frac 12-H}_0$ and
$\frac{1}{4}<H<\frac{1}{2}$. Fix $T>0$. Then there is a  unique
process $u$ in the space $\cv_{T}(H)$ such that for all $t\in
[0,T]$,
\begin{equation}
u(t,\cdot)=p_t* u_0  + \int_0^t \int_{\mathbb{R}}p_{t-s}(\cdot-y)u(s,y) W(ds,dy).
\end{equation}
\end{theorem}

\begin{proof}
The proof follows from the standard Picard iteration scheme, where we just set $u_{n+1}=\gga(u_{n})$. Details are left to the reader for sake of conciseness.
\end{proof}

\subsection{Existence and uniqueness via chaos expansions}\label{subsec: chaos}

Next, we provide another way to prove the existence and uniqueness of the solution to equation \eref{spde}, by means of chaos expansions. This will enable us to  obtain moment estimates.
 Before stating our main theorem in this direction, let us label an elementary lemma borrowed from \cite{HHNT} for further use. 

\begin{lemma}\label{lem:intg-simplex}
For $m\ge 1$ let $\alpha \in (-1+\ep, 1)^m$  with $\ep>0$ and  set $|\alpha |= \sum_{i=1}^m
\alpha_i  $. For $t\in\ott$, the $m$-th  dimensional simplex over $\ot$ is denoted by
$T_m(t)=\{(r_1,r_2,\dots,r_m) \in \R^m: 0<r_1  <\cdots < r_m < t\}$.
Then there is a constant $c>0$ such that
\[
J_m(t, \alpha):=\int_{T_m(t)}\prod_{i=1}^m (r_i-r_{i-1})^{\alpha_i}
dr \le \frac { c^m t^{|\alpha|+m } }{ \Gamma(|\alpha|+m +1)},
\]
where by convention, $r_0 =0$.
\end{lemma}

Let us now state a new existence and uniqueness theorem for our equation of interest.

\begin{theorem}\label{thm:exist-uniq-chaos}
Suppose that $\frac 14 <H<\frac 12$ and that the initial condition $u_0$ satisfies
\begin{equation}\label{cond:fu0}
\int_{\RR}(1+|\xi|^{\frac{1}{2}-H})|\mathcal{F}u_0(\xi)|d\xi < \infty\,.
\end{equation}
Then there exists a unique    solution to equation \eqref{spde},
that is,  there is  a unique process $u$ %(remember that $\laa_H$ is defined in Proposition \ref{prop:intg-wrt-W}) 
such that  $p_{t-\cdot}(x-\cdot)u$ is Skorohod integrable for any
$(t,x)\in\ott\times\R$ and relation  \eqref{eq:mild-formulation sigma}
holds true. 
\end{theorem}

\begin{remark}
(i) The formulation of Theorem \ref{thm:exist-uniq-chaos} yields the definition of our solution $u$ for all $(t,x)\in\ott\times\R$. This is in contrast with Theorem \ref{thm:exist-uniq-picard} which gives a solution sitting in $\dot{H}^{\frac12-H}_0$ for every value of $t$, and thus defined a.e. in $x$ only.
(ii) Condition \eqref{cond:fu0} is satisfied by constant functions.
\end{remark}

\begin{remark}
In the later paper \cite{HLN}, the existence and uniqueness in Theorem \ref{thm:exist-uniq-chaos} is obtained under a more general initial condition. Since the proof of Theorem \ref{thm:exist-uniq-chaos} for condition \eqref{cond:fu0} is easier and shorter, we present the proof as follows. 
\end{remark}    

\begin{proof}[Proof of Theorem \ref{thm:exist-uniq-chaos}]

Suppose that $u=\{u(t,x), \, t\geq 0, x \in \R^d\}$ is a solution to equation~\eqref{eq:mild-formulation sigma} in $\laa_{H} $. Then according to \eref{eq:chaos-dcp}, for any fixed $(t,x)$ the random variable $u(t,x)$ admits the following Wiener chaos expansion
\begin{equation}\label{eq:chaos-expansion-u(tx)}
u(t,x)=\sum_{n=0}^{\infty}I_n(f_n(\cdot,t,x))\,,
\end{equation}
where for each $(t,x)$, $f_n(\cdot,t,x)$ is a symmetric element in
$\HH^{\otimes n}$. %Furthermore, we have seen that It\^o and Skorohod's integral coincide for processes in $\laa_{H}$. 
Hence, thanks to~\eqref{eq:delta-u-chaos} and using an iteration procedure, one can find an
explicit formula for the kernels $f_n$ for $n \geq 1$. Indeed, we have:
\begin{multline}\label{eq:expression-fn}
f_n(s_1,x_1,\dots,s_n,x_n,t,x)\\
=\frac{1}{n!}p_{t-s_{\si(n)}}(x-x_{\si(n)})\cdots p_{s_{\si(2)}-s_{\si(1)}}(x_{\si(2)}-x_{\si(1)})
p_{s_{\si(1)}}u_0(x_{\si(1)})\,,
\end{multline}
where $\si$ denotes the permutation of $\{1,2,\dots,n\}$ such that $0<s_{\si(1)}<\cdots<s_{\si(n)}<t$
(see, for instance,  formula (4.4) in \cite{HN} or  formula (3.3) in \cite{HHNT}).
Then, to show the existence and uniqueness of the solution it suffices to prove that for all $(t,x)$ we have
\begin{equation}\label{chaos}
\sum_{n=0}^{\infty}n!\|f_n(\cdot,t,x)\|^2_{\HH^{\otimes n}}< \infty\,.
\end{equation}
The remainder of the proof is devoted to prove  relation \eqref{chaos}.

Starting from relation \eqref{eq:expression-fn}, some elementary Fourier computations show that
\begin{align*}
\cf f_n(s_1,\xi_1,\dots,s_n,\xi_n,t,x)&=
\frac{c_{H}^n}{n!}  \int_\RR
\prod_{i=1}^n e^{-\frac{\kappa}{2}(s_{\si(i+1)}-s_{\si(i)})|\xi_{\si(i)}+\cdots +
\xi_{\si(1)} -\zeta|^2} \\
&\quad\times { e^{-ix (\xi_{\sigma(n)}+ \cdots + \xi_{\sigma(1)}-\zeta)}} \mathcal{F}u_0(\zeta) e^{-\frac {\ka s_{\sigma(1)}|\zeta|^2} 2} d\zeta,
\end{align*}
where we have set $s_{\si(n+1)}=t$.
Hence, owing to formula \eref{eq: H_0 element H prod} for the norm in $\HH$ (in its Fourier mode version), we have
\begin{multline}\label{eq:expression-norm-fn}
n!\| f_n(\cdot,t,x)\|_{\HH^{\otimes n}}^2 =\frac{c_H^{2n}
}{n!} \, \int_{[0,t]^n}\int_{\RR^n}\bigg|   \int_\RR \prod_{i=1}^n
e^{-\frac {\kappa}{2} (s_{\si(i+1)}-s_{\si(i)})|\xi_i+\cdots +\xi_1-\zeta |^2} { e^{-ix (\xi_{\sigma(n)}+ \cdots + \xi_{\sigma(1)}-\zeta)}} \\ 
\mathcal{F}u_0(\zeta) e^{-\frac {\kappa s_{\sigma(1)}|\zeta|^2} 2} d\zeta \bigg|^2 
 \times  \prod_{i=1}^n  |\xi_i |^{1-2H} d\xi ds\,,
\end{multline}
where $d\xi$ denotes $d\xi_1 \cdots d\xi_n$ and similarly for $ds$.
Then using the change of variable $\xi_{i}+\cdots + \xi_{1}=\eta
_{i}$, for all $i=1,2,\dots, n$ and a linearization of the above expression, we obtain
\begin{multline*}
n!\| f_n(\cdot,t,x)\|_{\HH^{\otimes n}}^2 = \frac{c_H^{2n}
}{n!}\int_{[0,t]^n}\int_{\RR^n}
 \int_{\RR^2}\prod_{i=1}^n
e^{-\frac{\kappa}{2}(s_{\si(i+1)}-s_{\si(i)})(|\eta_{i}-\zeta|^2+|\eta_i-\zeta^{\prime}|^2)} \mathcal{F}u_0(\zeta) \overline{\mathcal{F}u_0(\zeta^{\prime})} \\
\times { e^{ix(\zeta -\zeta')}}e^{-\frac{\kappa s_{\sigma(1)}(|\zeta|^2+|\zeta^{\prime}|^2)}{2}} \prod_{i=1}^n|\eta_{i}-\eta_{i-1}|^{1-2H} d\zeta d\zeta^{\prime} d\eta ds\,,
\end{multline*}
where we have set $\eta_{0}=0$. Then we  use Cauchy-Schwarz inequality and bound the term $\exp(-\kappa s_{\sigma(1)}(|\zeta|^2+|\zeta^{\prime}|^2)/2)$
by $1$ to get
\begin{multline*}
n!\| f_n(\cdot,t,x)\|_{\HH^{\otimes n}}^2 \le
\frac{c_H^{2n}}{n!}
 \int_{\RR^2} \left ( \int_{[0,t]^n} \int_{\RR^n} \prod_{i=1}^n
e^{- \kappa (s_{\si(i+1)}-s_{\si(i)})|\eta_{i}-\zeta|^2}\prod_{i=1}^n|\eta_{i}-\eta_{i-1}|^{1-2H}d\eta ds \right)^{\frac{1}{2}} \\
\times \left ( \int_{[0,t]^n} \int_{\RR^n} \prod_{i=1}^n
e^{- \kappa (s_{\si(i+1)}-s_{\si(i)})|\eta_{i}-\zeta^{\prime}|^2}\prod_{i=1}^n|\eta_{i}-\eta_{i-1}|^{1-2H}d\eta ds \right)^{\frac{1}{2}}
\left|\mathcal{F}u_0(\zeta)\right| \left|\mathcal{F}u_0(\zeta^{\prime})\right| d\zeta d\zeta^{\prime}.
\end{multline*}
Arranging the integrals again, performing the change of variables $\eta_{i}:=\eta_{i}-\zeta$ and invoking the trivial bound $|\eta_{i}-\eta_{i-1}|^{1-2H}\le |\eta_{i-1}|^{1-2H}+|\eta_{i}|^{1-2H}$, this yields
\begin{equation}\label{eq:bnd-fn-L2-1}
n!\| f_n(\cdot,t,x)\|_{\HH^{\otimes n}}^2 \le \frac{c_H^{2n}
}{n!} \Bigg(\int_{\RR} L_{n,t}^{\frac{1}{2}}(\zeta) \left
|\mathcal{F}u_0(\zeta)\right|d\zeta\Bigg)^2 ,
\end{equation}
where $L_{n,t}(\zeta)$ is 
\begin{equation*}
 \int_{[0,t]^n} \int_{\RR^n} \prod_{i=1}^n
e^{-\kappa (s_{\si(i+1)}-s_{\si(i)})|\eta_{i}|^2} (|\zeta|^{1-2H}+|\eta_1|^{1-2H})
\times \prod_{i=2}^n(|\eta_{i}|^{1-2H}+|\eta_{i-1}|^{1-2H})d\eta ds.
\end{equation*}
Let us expand the product $\prod_{i=2}^{n} (|\eta_{i}|^{1-2H}+|\eta_{i-1}|^{1-2H})$ in the integral defining $L_{n,t}(\zeta)$. We obtain an expression of the form $\sum_{\al\in D_{n}} \prod_{i=1}^{n} |\eta_{i}|^{\al_{i}}$, where $D_{n}$ is a subset of multi-indices of length $n-1$. The complete description of $D_{n}$ is omitted for sake of conciseness, and we will just use the following facts: $\text{Card}(D_{n})=2^{n-1}$ and for any $\al\in D_{n}$ we have
\begin{equation*}
|\al|\equiv \sum_{i=1}^{n} \alpha_i = (n-1)(1-2H),
\quad\text{and}\quad
\al_{i} \in \{0, 1-2H, 2(1-2H)\}, \quad i=1,\ldots, n.
\end{equation*}
This simple expansion yields the following bound
\begin{multline*}
L_{n,t}(\zeta)
\leq|\zeta|^{1-2H}\sum_{\alpha \in D_{n}} \int_{[0,t]^n} \int_{\RR^n} \prod_{i=1}^n
e^{-\kappa (s_{\si(i+1)}-s_{\si(i)})|\eta_{i}|^2} \prod_{i=1}^n |\eta_i|^{\alpha_i}d\eta ds\\
+\sum_{\alpha \in D_{n}} \int_{[0,t]^n}\int_{\RR^n}\prod_{i=1}^n e^{-\kappa (s_{\si(i+1)}-s_{\si(i)})|\eta_{i}|^2} |\eta_1|^{1-2H} \prod_{i=1}^n |\eta_i|^{\alpha_i}d\eta ds\,.
\end{multline*}
Perform the change of variable $\xi_{i}= (\kappa (s_{\sigma(i+1)}-s_{\sigma(i)}))^{1/2} \eta_{i}$ in the above integral, and notice that $\int_{\R} e^{- \xi^{2}} |\xi|^{\alpha_i}d\xi$ is bounded by a constant
for $\alpha_i>-1$. Changing the integral over $\ot^{n}$ into an integral over the simplex, we get
\begin{eqnarray*}
L_{n,t}(\zeta)&\leq& C |\zeta|^{1-2H} n! c_H^n \sum_{\alpha \in D
_{n}} {
\int_{T_n(t)}\prod_{i=1}^n
(\kappa(s_{i+1}-s_{i}))^{-\frac{1}{2}(1+\alpha_i)}ds.}\\
&&+C n! c_H^n \sum_{\alpha \in D
_{n}} {
\int_{T_n(t)}(\kappa(s_{2}-s_{1}))^{-\frac{2-2H+\alpha_1}{2}}\prod_{i=2}^n
(\kappa(s_{i+1}-s_{i}))^{-\frac{1}{2}(1+\alpha_i)}ds.}
\end{eqnarray*}
We observe that whenever $\frac{1}{4}< H < \frac{1}{2}$, we have $\frac12(1+\al_{i})<1$ for all $i=2,\ldots n$, and it is easy to see that $\alpha_1$ is at most $1-2H$ so $\frac{1}{2}(2-2H+\alpha_1)<1$.
 (The condition $H>1/4$ comes from the requirement that 
when $\alpha_1=1-2H$, we need 
$\frac{1}{2}(2-2H+\alpha_1)=\frac{1}{2}(3-4H)<1$.)
 Thanks to Lemma \ref{lem:intg-simplex} and recalling that $\sum_{i=1}^n\alpha_i = (n-1)(1-2H)$ for all $\al\in D_{n}$, we thus conclude that
\begin{equation*}
 L_{n,t}(\zeta)
\leq\frac{C (1+t^{\frac{1}{2}-H}\kappa^{\frac{1}{2}-H}|\zeta|^{1-2H})n! c^nc_H^n t^{nH}\kappa^{nH-n}}{\Gamma(nH+1)}\,.
\end{equation*}
Plugging this expression into \eqref{eq:bnd-fn-L2-1}, we end up with
\begin{equation}\label{eq:bnd-H-norm-fn}
n!\| f_n(\cdot,t,x)\|_{\HH^{\otimes n}}^2
\leq
\frac{C c_H^n c^n t^{nH}\kappa^{nH-n}}{\Gamma(nH+1)}\left(\int_{\RR}(1+t^{\frac{1}{2}-H}\kappa^{\frac{1}{2}-H}|\zeta|^{\frac{1}{2}-H})\left| \mathcal{F}u_0(\zeta)\right| d\zeta\right)^2.
\end{equation}
The proof of \eqref{chaos} is now easily completed thanks to the asymptotic behavior of the Gamma
function and our assumption of $u_0$, and this finishes the existence and uniqueness proof.
\end{proof}

\section{Moment formula and bounds}
\label{sec:Anderson.momentbounds}

In this section we derive the Feynman-Kac formula for the moments of the solution to equation \eqref{spde} and the upper and lower bounds for the moments of the solution to equation \eref{spde} which allow us to conclude on the intermittency of the solution. We proceed by first getting an approximation result for $u$, and then deriving the upper and lower bounds for the approximation.
\subsection{Approximation of the solution}
The approximation of the solution we consider is based on the following approximation of the noise $W$.
For each $\ep>0$ and $\vp\in  \HH $,  we define
 	\begin{equation}\label{eq:cov-W-epsilon}
	W_{\varepsilon}(\varphi)
	= \int_0^t \int_{\mathbb{R}} [\rho_{\ep}*\varphi](s,x)W(ds,dy)
	=\int_0^t \int_{\mathbb{R}}\int_{\mathbb{R}}\varphi(s,x)\rho_{\varepsilon}(x-y)W(ds,dy)dx\,,
	\end{equation}
	where $\rho_ \varepsilon (x)=(2\pi \varepsilon)^{-\frac{1}{2}} e^{-{x^2}/{(2\varepsilon)}}$. Notice that the covariance of $W_\ep$ can be read (either in Fourier or direct coordinates) as:
	\begin{eqnarray}\label{eq:ident-cov-W-ep}
	\be\left[W_{\varepsilon}(\varphi) W_{\varepsilon}(\psi) \right]
	&=&
	c_{1,H} \int_0^t \int_{\mathbb{R}}
	\mathcal{F}\varphi(s,\xi)\, \overline{\mathcal{F}\psi(s,\xi)} \, e^{-\varepsilon |\xi|^2} |\xi|^{1-2H} d\xi ds   \\
	&=&
	c_{1,H} \int_0^t \int_{\mathbb{R}}\int_{\mathbb{R}}\varphi(s,x)f_{\varepsilon}(x-y)\psi(s,y) \, dx   dy   ds,  \notag
	\end{eqnarray}
	for every $\varphi, \psi$ in $\HH$,	where $f_{\ep}$ is given by $f_{\varepsilon}(x)=
	\mathcal{F}^{-1}(e^{-\varepsilon |\xi|^2} |\xi|^{1-2H})$. In other
	words, $W_\ep$ is still a white noise in time but its space
	covariance is now given by $f_{\ep}$. Note that $f_{\varepsilon}$
	is a real  positive-definite function, but is not necessarily
	positive.
The noise $W_{\ep}$ induces an approximation to the mild formulation of equation \eref{spde}, namely
\begin{equation}\label{appr eq}
u_{\ep}(t,x)=p_t* u_0(x) + \int_0^t \int_{\mathbb{R}}p_{t-s}(x-y)u_{\ep}(s,y) \, W_{\ep}(ds,dy) ,
\end{equation}
where the integral is understood (as in Subsection \ref{sec:picard}) in the It\^o sense. We will start by a formula for the moments of $u_{\ep}$.

\begin{proposition}\label{prop:appro-moments}
Let $W_{\ep}$ be the noise defined by \eqref{eq:cov-W-epsilon}, and
assume $\frac{1}{4}<H<\frac{1}{2}$.
Assume $u_0$ is such that
$\int_{\RR}(1+|\xi|^{\frac{1}{2}-H})|\mathcal{F}u_0(\xi)|d\xi<
\infty$.  Then

\noindent
\emph{(i)} Equation \eqref{appr eq} admits a unique solution.

\noindent
\emph{(ii)} For any integer $n \geq 2$ and $(t,x)\in\ott\times\R$, we have
\begin{equation}\label{appro moment}
\be \left[ u^n_{\varepsilon}(t,x)\right] = { \be_B
\left[\prod_{j=1}^n u_0(x+B_{\kappa t}^j) \exp \left( c_{1,H} \sum_{1\leq j\neq k
\leq n} V_{t,x}^{\ep,j,k}\right)\right],}
\end{equation}
with
\begin{equation}\label{eq:def-V-tx-epsilon}
V_{t,x}^{\ep,j,k}
=
\int_0^t f_{\varepsilon}(B_{ \kappa r}^j-B_{\kappa r}^k)dr
=
\int_0^t \int_{\mathbb{R}}e^{-\varepsilon |\xi|^2} |\xi|^{1-2H} e^{i\xi (B_{\kappa r}^j-B_{\kappa r}^k)} \, d\xi dr .
\end{equation}
In formula \eqref{eq:def-V-tx-epsilon}, $\{ B^j; j=1,\dots,n\}$ is a family of $n$ independent standard Brownian motions  which are also independent of $W$ and $\be_{B}$ denotes  the expected value with respect to the randomness in $B$ only.

\noindent
\emph{(iii)} The quantity  $\be [u^n_{\varepsilon}(t,x)]$  is uniformly bounded in $\ep$. More generally, for any $a>0$ we have
\begin{equation*}
\sup_{\ep>0}
\be_B \left[ \exp \left( a \sum_{1\leq j\neq k \leq n} V_{t,x}^{\ep,j,k}\right)\right] 
\equiv c_{a}<\infty .
\end{equation*}
\end{proposition}
\begin{proof}
The proof of item (i) is almost identical to the proof of Theorem~\ref{thm:exist-uniq-chaos}, and is omitted for sake of conciseness. Moreover, in the proof of (ii) and (iii), we may take $u_0(x)\equiv 1$ for simplicity.

In order to check item (ii), set
\begin{equation}\label{eq:def-Atx}
A_{t,x}^{\varepsilon}(r,y)=
\rho_{\varepsilon}(B_{\kappa (t-r)}^x-y),
\quad\text{and}\quad
\alpha^{\varepsilon}_{t,x}=\|A^{\varepsilon}_{t,x}\|^2_{\HH}.
\end{equation}
Then one can prove, similarly to Proposition 5.2 in \cite{HN}, that $u_{\ep}$ admits a Feynman-Kac representation of the form
\begin{equation}\label{eq:feynman-u-ep}
u_{\varepsilon}(t,x)=\be_B \lc \exp \lp  W (
A_{t,x}^{\varepsilon})-\frac{1}{2}\alpha^{\varepsilon}_{t,x}\rp
\rc\,.
\end{equation}
Now fix an integer $n \geq 2$. According to \eqref{eq:feynman-u-ep} we have
\begin{equation*}
\be \lc  u^n_{\varepsilon}(t,x)\rc=\be_W \lc\prod_{j=1}^n
\be_B\lc \exp \lp   W(A^{\varepsilon,
B^j}_{t,x})-
\frac{1}{2}\alpha_{t,x}^{\varepsilon,B^j}\rp \rc \rc\,,
\end{equation*}
where for any $j=1,\dots,n$,  $A_{t,x}^{\varepsilon,B^j}$ and $\alpha_{t,x}^{\varepsilon,B^j}$ are evaluations of  \eqref{eq:def-Atx} using the Brownian motion $B^j$. Therefore,  since $W(A^{\varepsilon, B^j}_{t,x})$ is a Gaussian random variable conditionally on $B$, we obtain
\begin{eqnarray*}
\be \lc  u^n_{\varepsilon}(t,x)\rc&=&
\be_B \lc
\exp \lp\frac{1}{2}\|\sum_{j=1}^n A_{t,x}^{\varepsilon,B^j}\|^2_{\HH}
-\frac{1}{2}\sum_{j=1}^n \alpha_{t,x}^{\varepsilon,B^j}\rp\rc \notag\\
&=& \be_B \lc
\exp \lp\frac{1}{2}\|\sum_{j=1}^n A_{t,x}^{\varepsilon,B^j}\|^2_{\HH}
-\frac{1}{2}\sum_{j=1}^n \| A_{t,x}^{\varepsilon,B^j}\|^2_{\HH}\rp\rc   \notag\\
&=&\be_B \lc \exp \lp\sum_{1\leq i < j \leq n}\langle
A_{t,x}^{\varepsilon,B^i},
A_{t,x}^{\varepsilon,B^j}\rangle _{\HH}\rp\rc\,.
\end{eqnarray*}
The evaluation of $\langle A_{t,x}^{\varepsilon,B^i}, A_{t,x}^{\varepsilon,B^j}\rangle _{\HH}$ easily yields our claim \eqref{appro moment}, the last details being left to the patient reader.

Let us now prove item (iii), namely
\begin{equation}\label{appro moment finite}
\sup_{\varepsilon > 0} \sup_{t \in [0,T], x \in \mathbb{R}}
\be \lc  u^n_{\varepsilon}(t,x)\rc < \infty\,.
\end{equation}
To this aim, observe first that we have obtained an expression \eqref{appro moment} which does not depend on $x\in\R$, so that the $\sup_{t \in [0,T], x \in \mathbb{R}}$ in \eqref{appro moment finite} can be reduced to a $\sup$ in $t$ only. Next, still resorting to formula \eqref{appro moment}, it is readily seen that it suffices to show that for two independent Brownian motions $B$ and $\tilde{B}$, we have
\begin{equation}\label{eq:bnd-exp-F-t-epsilon}
\sup_{\varepsilon > 0, t\in [0,T]} \be_{B} \lc \exp \left (c \, F_t^{\varepsilon} \right)\rc <\infty,
\quad\text{with}\quad
F_t^{\varepsilon} \equiv
\int_0^t \int_{\mathbb{R}} e^{-\varepsilon |\xi|^2} |\xi|^{1-2H} e^{i \xi (B_{\kappa r}-\tilde{B}_{\kappa r})}d\xi dr,
\end{equation}
for any positive constant $c$.  In order to prove \eqref{eq:bnd-exp-F-t-epsilon}, we expand the exponential and write:
\begin{equation}\label{eq:moments-F-t-epsilon}
\be_{B} \lc \exp (c \, F_t^{\varepsilon})\rc
=\sum_{l=0}^{\infty}\frac{\be_{B} \lc (c \, F_t^{\varepsilon})^l\rc}{l!}\,.
\end{equation}
Next, we have
\begin{align*}
\be_{B} \lc\left( F_t^{\varepsilon}\right)^l\rc&=
\be_{B} \lc \int_{[0,t]^l} \int_{\RR^l}
\prod_{j=1}^l  e^{-i  \xi_j (B_{\kappa r_j}-\tilde{B}_{\kappa r_j})-\varepsilon  |\xi_j|^2} |\xi_j|^{1-2H} d\xi dr \rc \\&
\leq
\int_{[0,t]^l} \int_{\RR^l}
\prod_{j=1}^{l} e^{-\kappa (t-r_{\sigma(j)})|\xi_j+\dots+\xi_1|^2} \, |\xi_j|^{1-2H} \, d\xi dr\,,
\end{align*}
where $\sigma$ is the permutation on $\{1,2,\dots, l\}$ such that $t \geq r_{\sigma(l)} \geq \cdots \geq r_{\sigma(1)}$. We have thus gone back to an expression which is very similar to \eqref{eq:expression-norm-fn}. We now proceed as in the proof of Theorem \ref{thm:exist-uniq-chaos} to show that \eref{appro moment finite} holds true from equation \eqref{eq:moments-F-t-epsilon}.
\end{proof}

Starting from Proposition \ref{prop:appro-moments}, let us take limits in order to get the moment formula for the solution $u$ to equation~\eqref{spde}.

\begin{theorem}\label{THM moment}
Assume $\frac{1}{4}<H<\frac{1}{2}$ and  consider $n\ge 1$, $j,k\in\{1,\ldots,n\}$ with $j\ne k$.
For $(t,x)\in\ott\times\R$,  denote by   $V_{t,x}^{j,k}$  the limit  in $L^2(\Omega)$  as $\ep\rightarrow 0$  of
\begin{equation*}%
%V_{t,x}^{j,k}
%=
%{\hbox{\tiny $ L^{2}(\oom)-$}}
%\lim_{\ep\to 0} V_{t,x}^{\ep,j,k} (\hbox{the limit is taken in $L^2(\Omega)$}),
%\quad\text{with}\quad
V_{t,x}^{\ep,j,k}
=
\int_0^t \int_{\mathbb{R}}e^{-\varepsilon |\xi|^2} |\xi|^{1-2H} e^{i\xi (B_{ \kappa r}^j-B_{\kappa r}^k)}d\xi dr.
\end{equation*}
Then  $\be \lc  u^n_{\varepsilon}(t,x)\rc$ converges as $\varepsilon \to 0$ to $\be [u^n(t,x)]$, which is given by
\begin{equation}\label{moment}
\be[u^n(t,x)] ={ \be_{B}\left[ \prod_{j=1}^n u_0(B^j_{\kappa t}+x)\exp \left(
c_{1,H} \sum_{1\leq j \neq k \leq n} V_{t,x}^{j,k} \right)\right]\, .}
\end{equation}
\end{theorem}
We note that in a recent paper \cite{HLN}, the moment formula for general covariance function is obtained. However we present the proof here for the sake of completeness.
\begin{proof}

As in Proposition \ref{prop:appro-moments}, we will prove the theorem for $u_0 \equiv 1$ for simplicity. For
any $p\ge 1$ and $1\le j < k \le n$, we can easily prove that
$V_{t,x}^{\ep,j,k}$ converges in $L^{p}(\oom)$ to $V_{t,x}^{j,k}$
defined by
\begin{equation}\label{eq:def-Vtx-jk}
V_{t,x}^{j,k}
=
\int_0^t \int_{\mathbb{R}} |\xi|^{1-2H} e^{i\xi (B_{\kappa r}^j-B_{\kappa r}^k)}d\xi dr.
\end{equation}
Indeed, this is due to the fact that $e^{-\varepsilon |\xi|^2} |\xi|^{1-2H} e^{i\xi (B_{\kappa r}^j-B_{\kappa r}^k)}$ converges to $|\xi|^{1-2H} e^{i\xi (B_{\kappa r}^j-B_{\kappa r}^k)}$ in the $d\xi\otimes dr\otimes d\bp$ sense, plus standard uniform integrability arguments. Now, taking into account relation \eqref{appro moment},  Proposition \ref{prop:appro-moments} (iii), the fact that $V_{t,x}^{\ep,j,k}$ converges to $V_{t,x}^{j,k}$ in $L^{2}(\oom)$ as $\ep\to 0$, and the inequality $|e^{x}-e^{y}|\leq (e^x+e^y)|x-y|$, we obtain
\begin{eqnarray*}
&&\be_B\left|\exp \left(c_{1,H}\sum_{1\leq j\neq k \leq n}V_{t,x}^{\epsilon,j,k} \right)-\exp \left(c_{1,H}\sum_{1\leq j\neq k \leq n}V_{t,x}^{j,k} \right)\right|\\
&\leq&\sup_{\epsilon >0}2\left(\be_B\left|\exp \left(2c_{1,H}\sum_{1\leq j\neq k \leq n}V_{t,x}^{\epsilon,j,k} \right)+\exp \left(2c_{1,H}\sum_{1\leq j\neq k \leq n}V_{t,x}^{j,k} \right)\right|^2\right)^{\frac{1}{2}} \\
&&\quad \times \left(\be_B \left|c_{1,H}\sum_{1\leq j\neq k \leq n}V_{t,x}^{\epsilon,j,k}-c_{1,H}\sum_{1\leq j\neq k \leq n}V_{t,x}^{j,k}\right|^2\right)^{\frac{1}{2}}\,,
\end{eqnarray*}
which implies 
\begin{eqnarray}\label{eq:lim-moments-u-epsilon}
\lim_{\ep\to 0} \be \lc  u^n_{\varepsilon}(t,x)\rc
&=&
\lim_{\ep\to 0} \be_B \left[ \exp \left( c_{1,H} \sum_{1\leq j\neq k \leq n} V_{t,x}^{\ep,j,k}\right)\right]  \notag \\
&=&
\be_B \left[ \exp \left( c_{1,H} \sum_{1\leq j\neq k \leq n} V_{t,x}^{j,k}\right)\right].
\end{eqnarray}

To end the proof, let us now identify the right hand side of \eqref{eq:lim-moments-u-epsilon} with $\be [u^n(t,x)]$, where $u$ is the solution to equation \eqref{spde}. For $\ep,\ep'>0$ we write
\[
\be \lc  u_{\varepsilon}(t,x) \, u_{\varepsilon'}(t,x)  \rc=
\be_B \lc  \exp \lp\ \langle A^{\varepsilon,B^1}_{t,x} , A^{\varepsilon',
B^2}_{t,x}  \rangle_{\HH}\rp\rc\, ,
\]
where we recall that $A^{\varepsilon,B}_{t,x}$ is defined by
relation \eqref{eq:def-Atx}. As before we can show that this
converges as $\varepsilon, \varepsilon'$ tend to zero. So,
$u_{\varepsilon}(t,x)$ converges in $L^2$ to some limit $v(t,x)$,
and the limit is actually in  $L^p$ , for all $p \geq 1$. Moreover,
$\be [v^k(t,x)]$ is equal  to the right hand side
of~\eref{eq:lim-moments-u-epsilon}.   Finally,  for any smooth random
variable $F$ which is a linear combination of $W({\bf
1}_{[a,b]}(s)\varphi(x))$, where $\varphi$ is a $C^{\infty}$
function with compact support,  using the %fact that It\^o's and Skorohod's integrals coincide on the set $\laa_{H}$, plus the
duality relation \eqref{dual}, we have
\begin{equation}\label{eq:duality-u-varepsilon}
{  \be \lc F u_{\varepsilon}(t,x)\rc
 =\be \lc F\rc+\be \lc \langle  Y^{\ep} ,DF\rangle _{\HH}\rc,}
\end{equation}
where
\begin{equation*}
Y^{t,x}({s,z})= \lp \int_{\mathbb{R}}
p_{t-s}(x-y) \, p_{\varepsilon}(y-z) u_{\varepsilon} (s,y)\, dy \rp \1_{\ot}(s) .
\end{equation*}
Letting $\varepsilon$ tend to zero in equation
\eref{eq:duality-u-varepsilon}, after some easy calculation we
get
\begin{equation*}
\be [F v_{t,x}]= \be[ F]  +\be \lc \langle DF, v
p_{t-\cdot}(x-\cdot)\rangle_{\HH}\rc\,.
\end{equation*}
This equation is valid for any $F \in \mathbb{D}^{1,2}$ by
approximation. So the above equation implies that the process $v$
is the solution of equation \eqref{spde}, and by the uniqueness of
the solution we have $v=u$.
\end{proof}

\subsection{Intermittency estimates}
In this subsection we prove some upper and lower bounds on the moments of the solution which entail the intermittency phenomenon.

\begin{theorem}\label{thm:intermittency-estimates}
Let $\frac{1}{4}<H<\frac{1}{2}$, and consider the solution $u$ to
equation \eqref{spde}. For simplicity we assume that the initial condition is $u_0(x)\equiv 1$.  Let $n \geq 2$ be an integer, $x\in\R$ and
$t\ge 0$. Then there  exist
some positive  constants $c_{1},c_{2},c_{3}$ independent of
$n$, $t$ and $\kappa$ with
$0<c_{1}<c_{2}<\infty$ satisfying
\begin{equation}\label{eq:intermittency-bounds}
\exp (c_{1} n^{1+\frac{1}{H}}\kappa^{1-\frac{1}{H}}t)
\leq \be\lc u^n(t,x) \rc
\leq c_{3} \exp \big(c_{2} n^{1+\frac{1}{H}}\kappa^{1-\frac{1}{H}} t\big)\,.
\end{equation}
\end{theorem}

%\begin{remark}
%Observe that the upper bound in \eqref{eq:intermittency-bounds}   also holds  for a general coefficient $\si$ (see the forthcoming relation \eqref{eq:upp-bnd-Lp-u-sigma-general}). 
%\end{remark}

\begin{proof}[Proof of Theorem \ref{thm:intermittency-estimates}]
We divide this proof into upper and lower bound estimates.

\noindent \textit{Step 1: Upper bound.} Recall from equation
\eqref{eq:chaos-expansion-u(tx)} that for $(t,x)\in\R_{+}\times\R$,
$u(t,x)$ can be written as:
$u(t,x)=\sum_{m=0}^{\infty}I_m(f_m(\cdot,t,x))$. Moreover, as a
consequence of the hypercontractivity property on a fixed chaos we
have (see \cite[p. 62]{Nua})
\begin{equation*}
\|I_m(f_m(\cdot,t,x))\|_{L^{n}(\oom)}\leq
(n-1)^{\frac{m}{2}}\|I_m(f_m(\cdot,t,x))\|_{L^{2}(\oom)} \,,
\end{equation*}
and substituting  the above right hand side by the bound \eqref{eq:bnd-H-norm-fn}, we end up with
\begin{eqnarray*}
\|I_m(f_m(\cdot,t,x))\|_{L^{n}(\oom)}
\leq
n^{\frac{m}{2}}\|I_m(f_m(\cdot,t,x))\|_{L^{2}(\oom)}
\leq
\frac{c^{\frac{n}{2}}n^{\frac{m}{2}}t^{\frac{mH}{2}}\kappa^{\frac{Hm-m}{2}}}
{ \Gamma(mH/2+1) } \,.
\end{eqnarray*}
Therefore from by the asymptotic bound of Mittag-Leffler 
function     $\sum_{n\ge 0}x^{n}/\Gamma(\al n+1) \le c_1 \exp(c_2 x^{1/a})$ (see  
\cite{kilbas}, Formula (1.8.10)),  we get:
\begin{eqnarray*}
\|u(t,x)\|_{L^{n}(\oom)}
\leq
\sum_{m=0}^{\infty} \|J_m(t,x)\|_{L^{n}(\oom)}
\leq
\sum_{m=0}^{\infty}\frac{c^{\frac{m}{2}}n^{\frac{m}{2}}t^{\frac{mH}{2}}\kappa^{\frac{Hm-m}{2}}}{\big(\Gamma(m
H+1)\big)^{\frac{1}{2}}}\leq c_{1}\exp {\big(c_{2} t n^{\frac{1}{H}} \kappa^{\frac{H-1}{H}}\big)}\,,
\end{eqnarray*}
from which the upper bound in our theorem is easily deduced.

\noindent
\textit{Step 2: Lower bound for $u_{\ep}$.}
For the lower bound, we start from the moment formula \eref{appro moment} for the approximate solution, and write 
\begin{multline*}
\be \lc  u^n_{\varepsilon}(t,x)\rc \\
=
\be_{B} \lc \exp \left(c_{1,H}\left[ \int_0^t \int_{\mathbb{R}} e^{-\varepsilon |\xi|^2}
\left| \sum_{j=1}^n e^{-i B_{\kappa r}^j \xi}\right|^2 |\xi|^{1-2H} d\xi dr
-nt \int_{\mathbb{R}} e^{-\varepsilon |\xi|^2} |\xi|^{1-2H} d\xi\right] \right)\rc.
\end{multline*}
In order to estimate the expression above, notice first that the
obvious change of  variable $\la= \ep^{1/2}\xi$ yields
$\int_{\mathbb{R}} e^{-\varepsilon |\xi|^2} |\xi|^{1-2H} d\xi=C
\ep^{-(1-H)}$ for some constant $C$.  Now for an additional arbitrary parameter $\eta>0$,
consider the set
\begin{equation*}
A_\eta=\left\{\om; \, \sup_{1\leq j\leq n}\sup_{0\leq r \leq t}|B_{\kappa r}^{j}(\om)|\leq
\frac{\pi}{3\eta}\right\}.
\end{equation*}
Observe that classical small balls inequalities for a Brownian motion (see (1.3) in \cite{LS}) yield $\bp(A_{\eta})\geq c_{1} e^{-c_{2}  \eta^2 n \kappa t}$ for a large enough $\eta$. In addition, if we assume that $A_{\eta}$ is realized and $|\xi|\le\eta$, some elementary trigonometric identities show that the following deterministic bound hold true: $| \sum_{j=1}^n e^{-i B_{\kappa r}^j \xi}| \ge \frac{n}{2}$.
Gathering those considerations, we thus get
\begin{align*}
\be \lc  u^n_{\varepsilon}(t,x)\rc
&\geq
\exp \left( c_1 n^2 \int_0^t \int_0^{\eta} e^{-\varepsilon |\xi|^2} |\xi|^{1-2H} d\xi dr - c_2 nt \varepsilon^{H-1} \right)
\bp\lp A_\eta \rp \\
&\geq C
\exp \left( c_1 n^2 t  \varepsilon^{-(1-H)} \int_0^{\ep^{1/2}\eta} e^{- |\xi|^2} |\xi|^{1-2H} d\xi  - c_2 nt \varepsilon^{-(1-H)} - c_{3} n \kappa t \eta^{2} \right).
\end{align*}
We now choose the parameter $\eta$ such that $\kappa \eta^2=\varepsilon^{-(1-H)}$, which means in particular that $\eta \to \infty$ as $\varepsilon \to 0$. It is then easily seen that $\int_0^{\ep^{1/2}\eta} e^{- |\xi|^2} |\xi|^{1-2H} d\xi$ is of order $\ep^{H(1-H)}$ in this regime, and some elementary algebraic manipulations entail
\begin{equation*}
\be \lc  u^n_{\varepsilon}(t,x)\rc
\geq C
\exp \left( c_1 n^2 t \kappa^{H-1}\varepsilon^{-(1-H)^2} -c_2 nt\varepsilon^{-(1-H)}\right)
\geq C \exp \left(c_{3} t \kappa^{1-\frac{1}{H}}n^{1+\frac{1}{H}}\right),
\end{equation*}
where the last inequality is obtained by choosing $\varepsilon^{-(1-H)}=c \, \kappa ^{\frac{H-1}{H}}n^{\frac{1}{H}}$ in order to optimize the second expression. We have thus reached the desired lower bound in \eqref{eq:intermittency-bounds} for the approximation $u^{\ep}$ in the regime $\varepsilon=c \, \kappa ^{\frac{1}{H}}n^{-\frac{1}{H(1-H)}}$.

\noindent
\textit{Step 3: Lower bound for $u$.}
To complete the proof, we need to show that for all sufficiently small $\varepsilon$, $\be \lc  u^n_{\varepsilon}(t,x)\rc\leq \be[u^n(t,x)]$. We thus start from equation \eref{appro moment} and use the series expansion of the exponential function as in \eqref{eq:moments-F-t-epsilon}. We get
\begin{equation}\label{eq:expansion-moment-u-epsilon}
\be \lc  u^n_{\varepsilon}(t,x)\rc=
\sum_{m=0}^{\infty} \frac{  c_{1,H}^m}{m!}  \,
\be _{B} \!\lc \left( \sum_{1\leq j \neq k \leq n} V_{t,x}^{\ep,j,k} \right)^m \rc,
\end{equation}
where we recall that $V_{t,x}^{\ep,j,k}$ is defined by \eqref{eq:def-V-tx-epsilon}. Furthermore, expanding the $m$th power above, we have
\begin{equation*}
\be_{B} \!\lc \left( \sum_{1\leq j \neq k \leq n} V_{t,x}^{\ep,j,k} \right)^m \rc
=
\sum_{\al\in K_{n,m}}  \int_{[0,t]^m} \int_{\mathbb{R}^m}
e^{-\varepsilon \sum_{l=1}^m |\xi_l|^2} \be_{B} \lc e^{i B^{\al}(\xi)} \rc
\prod_{l=1}^m |\xi_l|^{1-2H} \, d\xi dr\,,
\end{equation*}
where $K_{n,m}$ is a set of multi-indices defined by
\begin{equation*}
K_{n,m}=
\lcl
\al=(j_{1},\ldots,j_{m},k_{1},\ldots,k_{m}) \in \{1,\ldots,n\}^{2m} ; \,
j_{l}<k_{l} \text{ for all } l=1,\ldots,n
\rcl,
\end{equation*}
and $B^{\al}(\xi)$ is a shorthand for the linear combination $\sum_{l=1}^m \xi_{l}(B_{\kappa r_{l}}^{j_{l}}-B_{\kappa r_{l}}^{k_{l}})$. The important point here is that $E _{B} e^{iB^{\al}(\xi)}$ is positive for any $\al\in K_{n,m}$. We thus get the following inequality, valid for all $m\ge 1$
\begin{eqnarray*}
\be _{B} \!\lc \left( \sum_{1\leq j \neq k \leq n} V_{t,x}^{\ep,j,k} \right)^m \rc
&\le&
\sum_{\al\in K_{n,m}}  \int_{[0,t]^m} \int_{\mathbb{R}^m}
\be _{B} \lc e^{i B^{\al}(\xi)} \rc
\prod_{l=1}^m |\xi_l|^{1-2H} \, d\xi dr \\
&=&
\be _{B} \!\lc \left( \sum_{1\leq j \neq k \leq n} V_{t,x}^{j,k} \right)^m \rc,
\end{eqnarray*}
where $V_{t,x}^{j,k}$ is defined by \eqref{eq:def-Vtx-jk}. Plugging this inequality back into \eqref{eq:expansion-moment-u-epsilon} and recalling expression \eqref{moment} for $\be [u^n(t,x)]$, we easily deduce that $\be [u^n_{\ep}(t,x)] \le \be[u^n(t,x)]$, which finishes  the proof.
\end{proof}

\begin{minipage}{1.0\textwidth}
%\address{
\vskip 1cm
Yaozhong Hu, Jingyu Huang, Khoa L\^e and David Nualart: Department of Mathematics, University of Kansas, 405 Snow Hall, Lawrence, Kansas, 66044, USA. 
%}
%\email{

{\it E-mail address:} yhu@ku.edu, jhuang@math.utah.edu, khoa.le@ucalgary.ca,  nualart@ku.edu
%}
\vskip 0.5cm
%\address{
Samy Tindel: Department of Mathematics, Purdue University, West Lafayette, IN 47907, USA.
%}
%\email{

{\it E-mail address:} stindel@purdue.edu%}
\end{minipage}

\end{document}